\documentclass[draft,12pt]{article}
\usepackage{amsmath}
\usepackage{amsthm}
\usepackage{amssymb}
\usepackage[final]{graphicx}
\usepackage{color}
\usepackage{epsfig}
\newtheorem{defn}{Definition}[section]
\newtheorem{lemma}[defn]{Lemma}

\newtheorem{theorem}{Theorem}

\newtheorem{prop}[defn]{Proposition}

\newcommand{\aand}{\textrm{ and }}
\newcommand{\bdm}{\begin{displaymath}}
\newcommand{\edm}{\end{displaymath}}
\newcommand{\beqn}{\begin{equation}}
\newcommand{\eeqn}{\end{equation}}

\newcommand{\bbb}{\textrm{   }}
\newcommand{\bb}{\textrm{  }}

\newcommand{\bbbbb}{\bbb \bbb \bbb \bbb \bbb \bbb \bbb \bbb \bbb \bbb}
\newcommand{\BBb}{\bbbbb \bbbbb}

\newcommand{\Rm}{\mathbf{R}}
\newcommand{\Rmo}{\mathbf{\mathcal{R}}}

\newcommand{\bqr}{\begin{eqnarray*}} 
\newcommand{\eqr}{\end{eqnarray*}}

\newcommand{\Ra}{\mathbf{R_{1}}}  
\newcommand{\Rbb}{\mathbf{R_{2}}}  
\newcommand{\T}{\mathbf{\mathcal{T}}}

\newcommand{\pp}{\mathbf{P}} 
\newcommand{\Qm}{\mathbf{Q}} 
 
\newcommand{\pa}{\mathbf{P_{1}}} 
\newcommand{\pb}{\mathbf{P_{2}}} 
\newcommand{\ppo}{\mathbf{\mathcal{P}}} 
\newcommand{\Qmo}{\mathbf{\mathcal{Q}}} 
\newcommand{\Qtil}{\mathbf{\tilde{\mathcal{Q}}}} 
\newcommand{\Qatil}{\mathbf{\tilde{\mathcal{Q}}_{1}}} 
\newcommand{\Qbtil}{\mathbf{\tilde{\mathcal{Q}}_{2}}}
 
\newcommand{\QQm}{\mathbf{Q^{'}}} 
\newcommand{\pao}{\mathbf{\mathcal{P}_{1}}} 
\newcommand{\pbo}{\mathbf{\mathcal{P}_{2}}}

\newcommand{\Qpo}{\mathbf{\overline{\mathcal{Q}}}}
\newcommand{\QQpo}{\mathbf{\overline{\mathcal{QQ}}}}

\newcommand{\Ols}{\mathbf{P_{loops}}}
\newcommand{\Oloops}{\mathbf{\mathcal{P}_{loops}}}

\newcommand{\DOgQtilploops}{\mathbf{D_{orbit_{g}}}( \mathbf{\tilde{\mathcal{Q}}},\mathbf{\mathcal{P}_{loops}})}
\newcommand{\DOgQatilploops}{\mathbf{D_{orbit_{g}}}( \mathbf{\tilde{\mathcal{Q}}_{1}},\mathbf{\mathcal{P}_{loops}})}
\newcommand{\DOgQbtilploops}{\mathbf{D_{orbit_{g}}}( \mathbf{\tilde{\mathcal{Q}}_{2}},\mathbf{\mathcal{P}_{loops}})}
\newcommand{\DOgPQtil}{\mathbf{D_{orbit_{g}}}( \mathbf{\mathcal{P}}, \mathbf{\tilde{\mathcal{Q}}})}
\newcommand{\DOgPaQatil}{\mathbf{D_{orbit_{g}}}( \mathbf{\mathcal{P}_{1}}, \mathbf{\tilde{\mathcal{Q}}_{1}})}
\newcommand{\DOgPbQbtil}{\mathbf{D_{orbit_{g}}}( \mathbf{\mathcal{P}_{2}}, \mathbf{\tilde{\mathcal{Q}}_{2}})}

\newcommand{\DgPaPb}{\mathbf{D_{g}}( \mathbf{P_{1}},\mathbf{P_{2}})}
\newcommand{\DgPaRa}{\mathbf{D_{g}}( \mathbf{P_{1}},\mathbf{R_{1}})}
\newcommand{\DgPbRbb}{\mathbf{D_{g}}( \mathbf{P_{2}},\mathbf{R_{2}})}
\newcommand{\DgRaRbb}{\mathbf{D_{g}}( \mathbf{R_{1}},\mathbf{R_{2}})}

\newcommand{\DOgPaPb}{\mathbf{D_{orbit_{g}}}( \mathbf{\mathcal{P}_{1}},\mathbf{\mathcal{P}_{2}})}

\newcommand{\DOgsigma}{\mathbf{D}( \mathcal{O}(\Sigma_{g}))}

\newcommand{\Sg}{\Sigma_{g}}

\newcommand{\Modgsigma}{\mathbf{Mod}(\Sigma_{g})}
\newcommand{\PSg}{\mathcal{P}(\Sigma_{g})}
\newcommand{\OSg}{\mathcal{O}(\Sigma_{g})}
\newcommand{\OSf}{\mathcal{O}(\Sigma_{g-1})}
\newcommand{\kGcage}{(k,G)\textrm{--cage}}
\newcommand{\cGcage}{(3,G)\textrm{--cage}}
\begin{document}
\title{Upper Bound on Distance in the Pants Complex}
\author{Harriet Handel Moser}
\maketitle
% \input{coverpagehm}
%\newpage
\pagenumbering{arabic}

\section{Introduction} \label{intro}
In 1980 Hatcher and Thurston~\cite{HT} defined the pants graph, and showed that it is path connected.  In this graph vertices are pants decompositions and an edge joins 2 vertices if there is an elementary move between the 2 corresponding pants decompositions.  An elementary move occurs when one pants decomposition is transformed into another by replacing one curve with another one so that the two curves have minimal intersection(see Sec~\ref{pg}).  Distance comes from the path metric whereby the distance between 2 vertices joined by an edge is 1.  Approximately 20 years later, Hatcher, Lochak and Schneps~\cite{HLS} defined the pants complex by adding $2 $--cells to the 1--dimensional pants graph with distance the same as in the 1--skeleton. 

The purpose of this paper is to establish an upper bound on the distance between $2$ distinct pants decompositions in the pants complex.  Graph theory is used to accomplish this.  As part of the process, an upper bound on distance in the pants complex, modulo the action of the mapping class group, is found.  
A pants decomposition, $\pp$, will have a pants decomposition graph where the vertices are the pairs of pants resulting from cutting the surface along the curves of $\pp$, and the edges are the curves of $\pp$.  For a genus $g$ closed surface this is a trivalent graph on $2g - 2$ vertices with $3g - 3$ edges.  An elementary move from $\pp$ to another pants decomposition, $\Qm$, induces a transformation from the pants decomposition graph for $\pp$ to the pants decomposition graph for $\Qm$, called an elementary shift.  It can be shown that these $2$ pants decomposition graphs are isomorphic if and only if the action of the mapping class group on the surface takes $\pp$ to $\Qm$~\cite{P,UW}.  Thus each orbit of this action will have a unique trivalent graph associated to it.  The result is we can define a new graph, which we call the \emph{orbit graph}, that is the pants graph modulo the mapping class group.  See Figure~\ref{og}.  In this graph, vertices are orbits of the action of the mapping class group on the pants graph and an edge corresponds to an elementary shift from the trivalent graph of one orbit to that of another orbit.  The importance of the orbit graph is that a path in this graph from the orbit containing $\pp$ to the orbit containing $\Qm$ corresponds to a path of the same length from $\pp$ to $\Qm$ in the pants graph.  Section~\ref{pg}, ``Graphs", is devoted to a detailed discussion of several of the above mentioned graphs associated with pants decompositions.  The concept of an orbit graph is new to this paper.  In Section~\ref{distorb}, ``Distance in the Orbit Graph", principles of graph theory are applied to the orbits to find paths between orbits and their length.  From now on, $\Sg$ represents a closed surface of genus $g \geq 2$.  Our first result, Theorem~\ref{orbdiam}, gives an upper bound on the diameter of the orbit graph.   
\begin{theorem} \label{orbdiam}
Let $\DOgsigma$ denote the diameter of the orbit graph for $\Sg$.  Then
\begin{equation} \nonumber
   \DOgsigma  \leq \left\{\begin{array}{ll} 
                  4\log{(g-1)!} + 4g - 6  & \mbox{$2 \leq g \leq 5$}\\
                  4\log{(g-1)!} + 6g - 16 & \mbox{$6 \leq g$}
                 \end{array}
                \right.
\end{equation}
\end{theorem} 
There is a special pants decomposition, $\Ols$.  This is seen, along with its pants decomposition graph, $\Oloops$, in Figure~\ref{loops}.  Section~\ref{basedist}, ``Distance Within the Orbit $\Oloops$", finds a path between any $2$ pants decompositions that are in $\Oloops$ and then calculates the length of the path.  It makes use of a set of curves, $\{c_{1}, \ldots, c_{3g-1}\}$, shown in Figure~\ref{lick}, and twists about them, $\{T_{c_{1}}, \ldots,T_{c_{3g-1}} \}$, which are a set of Lickorish generators for $\Modgsigma$.  We call these curves the \emph{standard Lickorish generating curves} and we refer to the twists as the \emph{standard Lickorish generators}.  We will denote two arbitrary pants decompositions on $\Sg$ by $\pa \aand \pb$ and the distance between them in the pants complex as $\DgPaPb$.  Our second result is Theorem~\ref{goal}, which finds an upper bound on distance in the pants complex.  
\begin{theorem} \label{goal}
There is a pants decomposition, $\Ols$, and two other pants decompositions, $\Ra$ and $\Rbb$, in the same orbit as $\Ols$ where $\Ra$(resp. $\Rbb$) is the end of a path from $\pa$(resp. $\pb$).  Let $w,h \in \Modgsigma$  such that $h(\Ra)= \Ols$ and $w(\Ra)=\Rbb$. Then there are some standard Lickorish generating curves $\{c_{i_{1}}, \ldots, c_{i_{k}} \}$ such that $w=T^{e_{i_{1}}}_{h^{-1}(c_{i_{1}})} \cdots T^{e_{i_{k}}}_{h^{-1}(c_{i_{k}})}$ with $e_{i_{j}}=\pm 1 $ for $1 \leq j \leq k$ where this is a minimum length word in these twists and
\begin{equation} \nonumber
   \DgPaPb  \leq \left\{\begin{array}{ll} 
                 4\log{(g-1)!} + 4g - 6 +\sum_{1 \leq i_{j} \leq g}|e_{i_{j}}| + \\
\bbb \bbb \bbb  4\sum_ {i_{j}=g+1,\, 2g-1}|e_{i_{j}}| + 6\sum_ {g+2 \leq i_{j} \leq 2g-2}|e_{i_{j}}| & \mbox{$2 \leq g \leq 5$}\\
               4\log{(g-1)!} + 6g - 16 +\sum_{1 \leq i_{j} \leq g}|e_{i_{j}}| + \\
\bbb \bbb \bbb  4\sum_ {i_{j}=g+1,\, 2g-1}|e_{i_{j}}| + 6\sum_ {g+2 \leq i_{j} \leq 2g-2}|e_{i_{j}}| & \mbox{$6 \leq g$}
                 \end{array}
                \right.
\end{equation}
\end{theorem} 
We conclude with Section~\ref{op}, ``Open Problems and Limitations".
\section{Graphs}  \label{pg}
We begin with a few definitions.  They can be found in the work of Margalit~\cite{M} and Putman~\cite{P}, as well as many others.  A \emph{pants decomposition} is a maximal collection of pairwise disjoint isotopy classes of simple closed curves that are essential and non-peripheral.  From now on, unless specified otherwise, the term curve will refer to the isotopy class of a curve.  So for the surface $\Sg$, a pants decomposition, $\pp$, consists of $3g - 3$ curves that cut the surface into $2g-2$ pairs of pants and is sometimes written as $\{c_{1},\ldots,c_{3g-3}\}$.  If we consider $\{c_{1},\ldots,c_{i-1},c_{i+1},\ldots,c_{3g-3}\}$, these curves cut the surface so that all pieces are the same as before except for the pairs of pants containing $c_{i}$.  This exception can be either a one--holed torus, $\T$, where $c_{i} \subset \T$, or a $4$--holed sphere where 2 pairs of pants have $c_{i}$ as a common curve separating them.  An \emph{elementary move} is a situation where the curve $c_{i}$ is replaced by another curve $c^{'}_{i}$ so that for the one--holed torus, $i(c_{i},c^{'}_{i})=1$, and for the $4$--holed sphere, $i(c_{i},c^{'}_{i})=2$.  The result is another pants decomposition $\{c_{1},\ldots,c_{i-1},c^{'}_{i},c_{i+1},\ldots,c_{3g-3}\}$.  This can be seen in Figure~\ref{em}.
\begin{figure}

\begin{center}

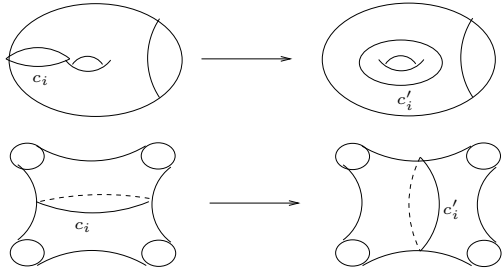

\caption{Elementary Moves} \label{em}
\end{center}
\end{figure}

A \emph{pants graph}, $\PSg$, for $\Sg$, is a $1$--dimensional complex where vertices are pants decompositions and $2$ vertices are joined by an edge when there is an elementary move between the $2$ pants decompositions associated to the vertices.  A \emph{path} in the pants graph is from vertex to vertex along these edges.  The \emph{length}  of a path is the number of edges traversed along the path and the \emph{distance} between $2$ vertices is the shortest path between them.  Let $\pp=\{c_{1},\ldots,c_{3g-3}\}$ be a pants decomposition for $\Sg$.  Then $\phi( \pp)$, the \emph{pants decomposition graph} for $\pp$, is a graph with $2g-2$ vertices, one for each pair of pants obtained by cutting the surface along the curves of the decomposition, and $3g-3$ edges, one for each curve in $\pp$.  A curve belonging to $2$ pairs of pants will correspond to an edge between the vertices for the pairs of pants, and a curve that is entirely in a pair of pants will correspond to a loop type edge at the vertex for that pair of pants.  Figure~\ref{base} demonstrates this.  The valence of every vertex will be $3$.  This is true even when a loop is attached at a vertex since a loop has $2$ arcs coming out of the vertex, and another edge goes to another vertex.  Such a graph is called a \emph{trivalent} or \emph{cubic} graph.      
\begin{figure}
\begin{center}

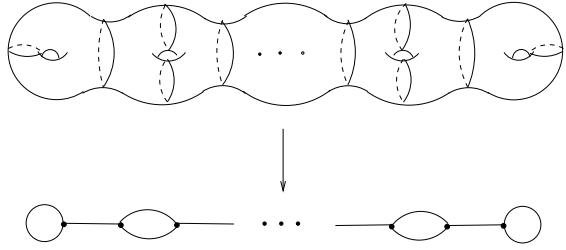

\caption{Pants Decomposition and Pants Decomposition Graph}  \label{base}
\end{center}
\end{figure}

There is a relationship between pants decomposition graphs that is similar to elementary moves between pants decompositions.  Let there be an elementary move between $\pa = \{c_{1},\ldots,c_{i-1},c_{i},c_{i+1},\ldots,c_{3g-3g}\}$ and $\pb = \{c_{1},\ldots,c_{i-1},c^{'}_{i},c_{i+1},\ldots,c_{3g-3}\}$, and $\phi(\pa) \aand \phi(\pb)$ their corresponding pants decomposition graphs.  If $i(c_{i},c^{'}_{i})=1$, then $c_{i} \aand c^{'}_{i}$ are curves in the same one--holed torus, so they both occur as a loop attached to the same vertex while the rest of $\phi(\pa) \aand \phi(\pb)$ remain unchanged and identical to each other.  Therefore, $\phi(\pa) \simeq \phi(\pb)$.  Using Putman's method~\cite{P} for changing $\phi(\pa)$ to $\phi(\pb)$ when $i(c_{i},c^{'}_{i})=2$, the following is done.  In $\phi(\pa)$, the edge corresponding to $c_{i}$ will be adjacent to two vertices, $v_{1} \aand v_{2}$. Collapse this edge so that $v_{1} \aand v_{2}$ become one vertex.  Four edges, $2$ edges from each of the vertices $v_{1} \aand v_{2}$, will now be adjacent to this resulting vertex.  Now expand this vertex to $2$ new vertices, with one edge between these $2$ new vertices, and each new vertex will have one edge from $v_{1}$ and another edge from $v_{2}$.  This transformation of $\phi(\pa)$ is called an \emph{elementary shift} \label{estext}.  It can be seen in Figure~\ref{es}.  Every elementary shift of $\phi(\pa)$ is the result of an elementary move of $\pa$.
\begin{figure}
\begin{center}

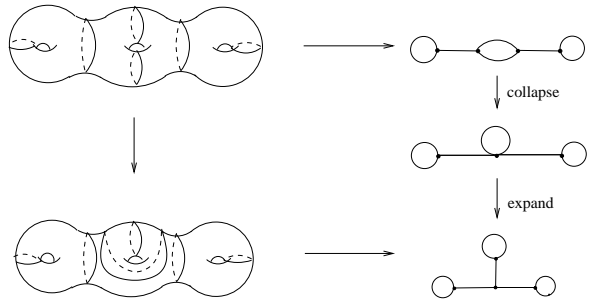

\caption{Elementary Shifts of Pants Decomposition Graphs}  \label{es}
\end{center}
\end{figure}

A homeomorphism of a surface preserves intersection numbers, so a homeomorphism of $\Sg$ takes a pants decomposition to another pants decomposition.  Consequently, $\Modgsigma$, the mapping class group of $\Sg$, operates on the pants graph, $\PSg$.  It has been shown that for any $2$ pants decompositions $\pa \aand \pb, \bb \phi(\pa)$ and $\phi(\pb)$ are isomorphic if and only if there is an $f \in \Modgsigma$ such that $f \cdot \pa = \pb$~\cite {P}~\cite{UW}.  Therefore, as the mapping class group acts on $\PSg$, each orbit of vertices has a unique trivalent graph on $2g - 2$ vertices that is the pants decomposition graph, modulo isomorphism, of every pants decomposition in the orbit.  From~\cite{UW} we see that the orbit of an edge in $\PSg$ between $\pa \aand \pb$, under the action of the mapping class group, consists of all edges in $\PSg$ that connect a pants decomposition in the orbit of $\pa$ to one in the orbit of $\pb$.  Thus the orbit of the edge corresponds to an elementary shift between $\phi(\pa)$ and $\phi(\pb)$.  This allows us to define a new graph, the \emph{orbit graph}, $\OSg$, where vertices are isomorphism classes of pants decomposition graphs and $2$ vertices are joined by an edge if there is an elementary shift between the $2$ corresponding pants decomposition graphs.  We can now refer to a specific pants decomposition graph by the same term as the orbit corresponding to it under the action.  Path and distance are defined as in all the other graphs discussed so far.  In Figure~\ref{og} we see the pants decomposition graphs for $10$ of the orbits for a genus $5$ surface and a part of the orbit graph where the $10$ vertices are these orbits.  Thus we see that the orbit graph is a graph of graphs.     
\begin{figure}
\begin{center}

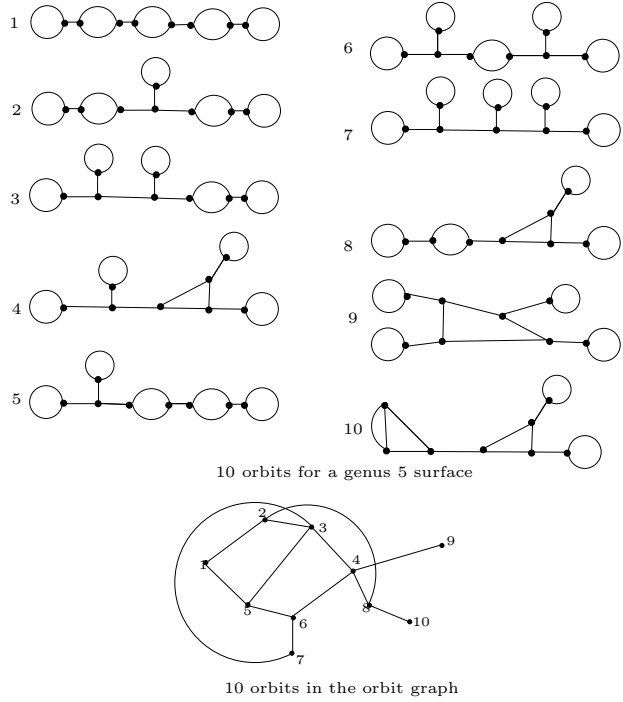

\caption{Orbits and the Orbit Graph}  \label{og}
\end{center}
\end{figure}

The pants graph, $\PSg$, is connected~\cite{HT}.  Therefore the orbit graph is connected, so there is a path connecting any $2$ vertices in the orbit graph.  Thus any pants decomposition graph can be transformed into any other one by a series of elementary shifts.  Let $\pp$ be a pants decomposition and $\ppo$ be it's orbit.  Then, if $\Rmo$ is another orbit that is one elementary shift away from $\ppo$, there will be some $\Rm \in \Rmo$ that is one elementary move away from $\pp$.  Let $\pp \aand \Qm$ be $2$  pants decompositions, with $\ppo \aand \Qmo$ their orbits.  Since there is a path from $\ppo$ to $\Qmo$ in the orbit graph, there will be a path in the pants graph from $\pp$ to some pants decomposition $\QQm \in \Qmo$, where each vertex in the pants graph path is a pants decomposition whose pants decomposition graph is an orbit corresponding to a vertex in the orbit path.  Consequently, if we can determine distance in the orbit graph, it would remain to find distance between $2$ pants decompositions in the same orbit.
\section{Distance in the Orbit Graph} \label{distorb}
The \emph{girth} of a graph is the length of the shortest cycle in the graph.  This section shows how to transform a pants decomposition graph into another one with a loop, based on the girth of the graph.  The calculation of an upper bound on girth enables us to find an upper bound on the number of elementary shifts needed to accomplish this transformation.  Using this information we then show how to find the distance between $2$ pants decomposition graphs. Graphs where all vertices have the same valence, like the trivalent pants decomposition graph for closed surfaces, has been the object of great interest to graph theorists.  As early as 1967 Tutte~\cite{T} discussed these graphs and introduced the term $\kGcage$. A $\kGcage$ is a graph with girth $G$ where all the vertices have valence $k$. The object of much of this study has been to see when such cages exist and how many vertices they have.  Tutte~\cite{T} showed there is a graph whenever $k=3$, and also calculated the number of vertices for some specific instances of girth.  However, for $G \geq 3 $, there is the following general lower bound, $LB$, on vertex number for all cages ~\cite {T,BB,W}:
\bdm
LB = \frac{2(k-1)^{G/2} - 2}{k-2} \qquad \textrm{for } G \textrm{ even}
\edm
and
\bdm
LB = \frac{k(k-1)^{\frac{G-1}{2}} - 2}{k-2} \qquad \textrm{for } G \textrm{ odd.}
\edm
Our pants decomposition graphs are $\cGcage$s, so for these graphs, where $G \geq 3$:
\bdm
LB = 2(2)^{G/2} - 2 \qquad \textrm{for } G \textrm{ even}
\edm
and
\bdm
LB = 3(2)^{\frac{G-1}{2}} - 2 \qquad \textrm{for } G \textrm{ odd}.
\edm
Since these graphs all have $2g - 2$ vertices, their girth will be such that $LB \leq 2g-2$.  For $G \geq 3$ and even, we must have $G \geq 4$, so $G/2 \geq 2$.  Then $2^{G/2} \geq 4 > 2$.  Thus $2(2)^{G/2} > (2)^{G/2} + 2$, so $LB > (2)^{G/2}$.  For $G$ odd,
\begin{eqnarray*}   
LB & = & 3(2)^{\frac{G-1}{2}} - 2 \\
    & = & \frac{3}{\sqrt{2}}(2)^{G/2} - 2\\
    & = & (\frac{3}{\sqrt{2}} -1)(2)^{G/2} - 2 + (2)^{G/2}.
\end{eqnarray*}
For $G \geq 3, \bb (2)^{G/2} \geq (2)^{3/2} = 2 \sqrt{2}$, so
\begin{eqnarray*}
(\frac{3}{\sqrt{2}} -1)(2)^{G/2} & \geq & (\frac{3}{\sqrt{2}} -1)2 \sqrt{2}\\
                                                  & =  &  6 - 2 \sqrt{2}\\ 
                                                  & > & 2.
\end{eqnarray*}
Therefore, $(\frac{3}{\sqrt{2}} -1)(2)^{G/2} - 2 > 0$, so
\bdm
LB = (\frac{3}{\sqrt{2}} -1)(2)^{G/2} - 2 + (2)^{G/2} > (2)^{G/2}.
\edm
Consequently, for any pants decomposition graph with girth $G \geq 3, \bb (2)^{G/2} < 2g-2$.  Taking logs, we find $G/2 < \log{(2g-2)}$.  Then,
\begin{eqnarray*} 
G  & < & 2\log{2(g-1)}\\
    & = & 2(1+ \log{(g-1)}).
\end{eqnarray*}
There are pants decomposition graphs that have girth of $2$, such as the genus $2$ orbit that does not have a loop.  See Figure~\ref{ls}.  Since we are interested only in $g \geq 2, \bb  2(1+ \log{(g-1)}) \geq 2$, so $G \leq 2(1+ \log{(g-1)})$ for all $G \geq 2 $.  Applying the ``collapse and expand" process associated with an elementary shift ( as seen in Figure~\ref{es} ) to an edge of a cycle reduces the length of the cycle by $1$.  So it would take one elementary shift less than the length of a cycle to reduce the cycle to a loop, and then $G-1$ elementary shifts would transform one pants decomposition graph, $\ppo$, to another one, $\Qmo$, which has a loop.  Thus, the upper bound on the distance between two such graphs, $\ppo \aand \Qmo$, in the orbit graph, would be $1+2 \log{(g-1)}$.  Putman introduces a technique~\cite{P} which we call \emph{loop surgery}.  We can perform loop surgery on $\Qmo$ by cutting off the loop, the vertex the loop is attached to, and the edge adjacent to that vertex.  The remaining vertex from that edge is then removed, leaving the other two edges that were attached to this last vertex unconnected to anything.  Now join these two edges to make one smooth new edge in a new trivalent graph which has $2$ less vertices.  See Figure~\ref{ls}.  
\begin{figure}
\begin{center}

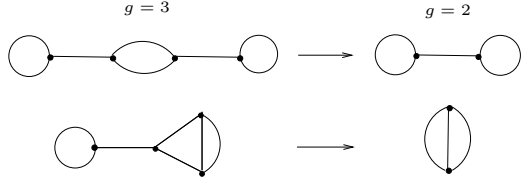

\caption{Loop Surgery}  \label{ls}
\end{center}
\end{figure}
The result is a pants decomposition graph, $\Qpo$, for a genus $g-1$ surface. Let $\mathbf{Y}$ denote the loop, the vertex the loop is attached to and the edge adjacent to that vertex.  
\begin{prop} \label{looppull}
For each pants decomposition graph, $\ppo$, there is another pants decomposition graph, $\Qtil$, with $g$ loops, where
\bdm
\DOgPQtil    \leq 2g - 3 + 2\sum_{2 \leq i \leq g}\log{(i-1)}.
\edm
\end{prop}
\begin{proof}
If $g = 2$, it takes at most $1$ elementary shift to transform $\ppo$ to a graph with $2$ loops, so the proposition is true for $g=2$.\
Using induction, assume true for $g-1$.
Let $g \geq 3$.  Transform $\ppo$ to $\Qmo$ and perform loop surgery on $\Qmo$ to get $\Qpo$ as described above.  We now find a girth length cycle in $\Qpo$ whose length will be less than or equal to $2 + 2\log{(g-2)}$.  Then $\Qpo$ can be transformed to a graph, $\QQpo$, with a loop, by doing at most $1 + 2\log{(g-2)}$ elementary shifts along all but one of the edges of the cycle.  This corresponds to a path in $\OSf$ of length $1 + 2\log{(g-2)}$.  We now lift this path to $\OSg$ by reattaching $\mathbf{Y}$ to $\Qpo$ to get $\Qmo$ and carrying $\mathbf{Y}$ along with the transformations.  If $g=3$, $\QQpo$ will be the genus $2$ graph with $2$ loops, so it is possible that in it's lift $\mathbf{Y}$ will be attached to one of these loops.  In this case, one more elementary shift is required to move $\mathbf{Y}$ to the edge between these $2$ loops.  Thus it is true for $g=3$.  Now let $g \geq 4$.  If none of the edges of the cycle in $\Qpo$ are the result of joining the two loose edges of $\ppo$ after the removal of  $\mathbf{Y}$, then the end result of the lift is that $1 + 2\log{(g-2)}$ elementary shifts will give us a graph in $\OSg$ that is the lift of $\QQpo$ and it has one more loop than $\Qmo$ has.  Otherwise, another vertex gets added to the cycle in $\Qpo$ at the point where $\mathbf{Y}$ is attached, so one more elementary shift must be done to get this new loop.  So we see that it takes at most $ 1 + 2\log{(g-1)} + 1 + 2\log{(g-2)} + 1$ elementary shifts to get a graph with $2$ more loops than $\ppo$ has.  The logic is that we can continue the process recursively by pulling off one new loop for each suceeding lower genus, going all the way down to $g=2$.  This is formalized by an induction argument.  By induction we know that it takes at most $2(g-1) - 3 + 2\sum_{2 \leq i \leq g-1}\log{(i-1)}$ elementary shifts to convert $\Qpo$ to a graph with $g-1$ loops.  The path of this transformation can be lifted to $\OSg$, starting at $\Qmo$.  It takes at most $1 + 2\log{(g-1)}$ elementary shifts to convert $\ppo$ to $\Qmo$.  Finally, if in the lift $\mathbf{Y}$ is attached to the transformation of $\Qpo$ along a loop edge, the attaching vertex of $\mathbf{Y}$ adds a vertex to the loop, creating a cycle of length $2$ and keeping the number of loops in the lift of $\QQpo$ at $g-1$.  One more elementary shift is needed to convert the cycle to another loop, giving us a graph with $g$ loops.  Denote this graph by $\Qtil$.  See Figure~\ref{anyloops}.  Then
\bqr
\DOgPQtil & \leq & 1 + 2\log{(g-1)} + 2(g-1) -3 + 2\sum_{2 \leq i \leq g-1}\log{(i-1)} +1\\
               & \leq &  2g - 3 + 2\sum_{2 \leq i \leq g}\log{(i-1)}. 
\eqr
So it is true for all $g \geq 2$.  
\end{proof}
\begin{figure}
\begin{center}

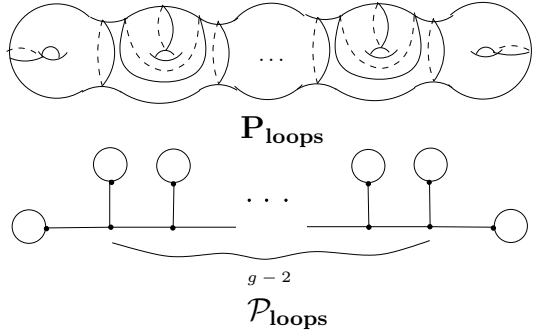

\caption{Special Pants Decomposition and Pants Decomposition Graph With $g$ Loops}\label{loops}
\end{center}
\end{figure}

We now want to find the distance from $\Qtil$ to $\Oloops$, a special trivalent graph on $2g - 2$ vertices with $g$ loops, as in Figure~\ref{loops}.  We do this because we know how to calculate distance between pants decompositions in the orbit $\Oloops$.  Our graphs are isomorphism classes of graphs, so when $2 \leq g \leq 5$ we have $\Qtil$ isomorphic to $\Oloops$, resulting in $\DOgQtilploops = 0$.  For $g \geq 6$, a graph like $\Qtil$ would look like the one in Figure~\ref{anyloops}.  It has $n$ branches, with $l_{i} + 1$ loops on the $i$th branch for $1 \leq i \leq n$.  An elementary shift according to the ``collapse and expand" process on the edge from vertex $v_{i}$ to vertex $v_{i,1}$ transforms the graph by moving the loop at the vertex $v_{i,1}$ to the edge between $v_{i-1}$ and a new vertex at the base of the $i$th branch and shortening the $i$th branch by $1$ loop.  So it would take $l_{i}$ elementary shifts to flatten the entire branch.  Therefore it would take $\sum_{1 \leq i \leq n} l_{i}$ elementary shifts to convert $\Qtil$ to $\Oloops$.  However, there are $l_{i} + 1$ loops on each branch vertex $v_{1}$ to $v_{n}$, so $g = \textrm{number of loops} = n + 2 + \sum_{1 \leq i \leq n} l_{i}$.  If $n$ is equal to either $1$ or $2$, $\Qtil$ is isomorphic to $\Oloops$, so $\DOgQtilploops = 0$.  Then if $n \geq 3$,   $\sum_{1 \leq i \leq n} l_{i} \leq g - 5$.  Thus
\beqn \label{dloops}
\DOgQtilploops \leq \max{(0,g -5)}
\eeqn
\begin{figure}
\begin{center}

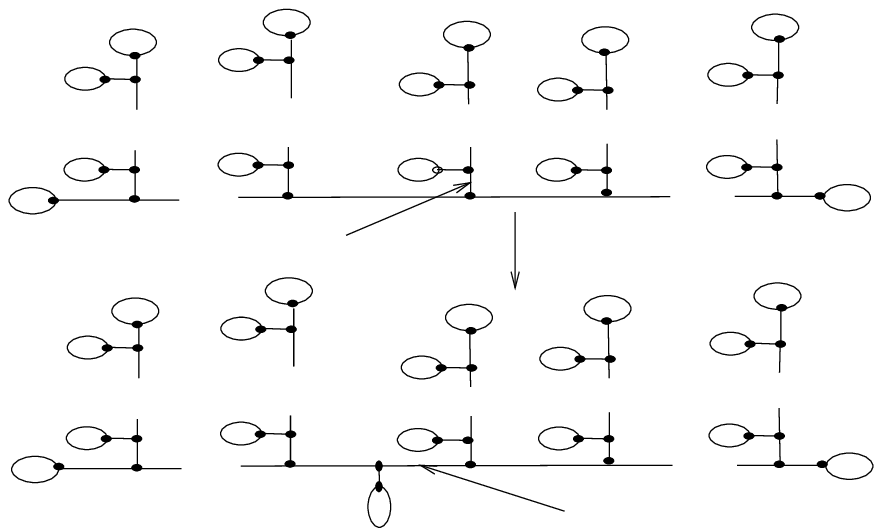

\caption{Move the Loop on $\Qtil$}  \label{anyloops}
\end{center}
\end{figure}
\begin{proof}[Proof of Theorem~\ref{orbdiam}] 
Let $\pao \aand \pbo$ denote the respective pants decomposition graphs for $\pa \aand \pb$ and $\Qatil \aand \Qbtil$ the respective transformations of $\pao \aand \pbo$ to graphs with $g$ loops.  Then
\begin{eqnarray}
\DOgPaPb & \leq & \DOgPaQatil + \DOgQatilploops + \nonumber \\
               &        &      \BBb  \DOgPbQbtil + \DOgQbtilploops \nonumber\\
               &  \leq  &  2(2g - 3 + 2\sum_{2 \leq i \leq g}\log{(i-1)} + \max{(0,g -5)}) \nonumber \\
               &  \leq  & \left\{\begin{array}{ll} 
                  4g - 6 +4\sum_{2 \leq i \leq g}\log{(i-1)}  & \mbox{$2 \leq g \leq 5$}\\
                  6g - 16 +4\sum_{2 \leq i \leq g}\log{(i-1)} & \mbox{$6 \leq g$}
                 \end{array}
                \right. 
\end{eqnarray}
The \emph{diameter} of the orbit graph for $\Sg$ is the upper bound on distance between any two vertices in this graph.  We know this exists because there are only a finite number of distinct trivalent graphs that have a given even number of vertices.  Since $\pa \aand \pb$ are arbitrary, 
\begin{equation} \nonumber
   \DOgsigma  \leq \left\{\begin{array}{ll} 
                  4\log{(g-1)!} + 4g - 6  & \mbox{$2 \leq g \leq 5$}\\
                  4\log{(g-1)!} + 6g - 16 & \mbox{$6 \leq g$}
                 \end{array}
                \right.
\end{equation}
\end{proof}
The process of finding an upper bound on the distance between $\pa \aand \pb$ involves constructing a pair of paths in the orbit graph from $\pao \aand \pbo$ to $\Oloops$ and calculating an upper bound on their length.  The paths will have corresponding paths of the same length in the pants graph from $\pa$(resp. $\pb$) to some $\Ra$(resp. $\Rbb$) in $\Oloops$.  Since 
\bdm 
\DgPaPb  \leq  \DgPaRa +\DgPbRbb +\DgRaRbb,
\edm
we have
\begin{equation} \label{diamglp}
   \DgPaPb  \leq \left\{\begin{array}{ll} 
                  4\log{(g-1)!} + 4g - 6 +\DgRaRbb & \mbox{$2 \leq g \leq 5$}\\
                  4\log{(g-1)!} + 6g - 16 +\DgRaRbb & \mbox{$6 \leq g$}
                 \end{array}
                \right.
\end{equation}
\section{Distance Within The Orbit $\Oloops$} \label{basedist}
Let $\{c_{1}, \ldots, c_{3g-1}\}$ be the standard Lickorish generating curves, so $\{T_{c_{1}},\ldots,T_{c_{3g-1}} \}$ will be the standard Lickorish generators.  See Figure~\ref{lick}.  There is a path in the pants graph from $\Ols$ to both $T_{c_{i}}(\Ols)$ and  $T^{-1}_{c_{i}}(\Ols)$ of length $l_{c_{i}}$ where:
\beqn \label{lc}
   l_{c_{i}} = \left\{\begin{array}{ll} 
                  1 & \mbox{$1 \leq i \leq g$}\\
                  4 & \mbox{$i=g+1,2g-1$}\\
                  6 & \mbox{$g+2 \leq i \leq 2g-2$}\\
                  0 & \mbox{$2g \leq i \leq 3g-1$}
                 \end{array}
                \right.
\eeqn
\begin{figure}

\begin{center}

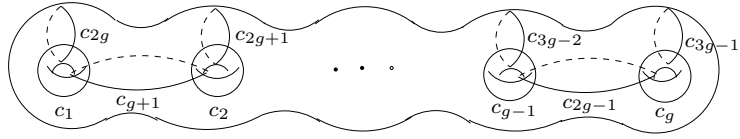

\caption{Standard Lickorish Generating Curves} \label{lick}
\end{center}
\end{figure}
\begin{figure}
\begin{center}

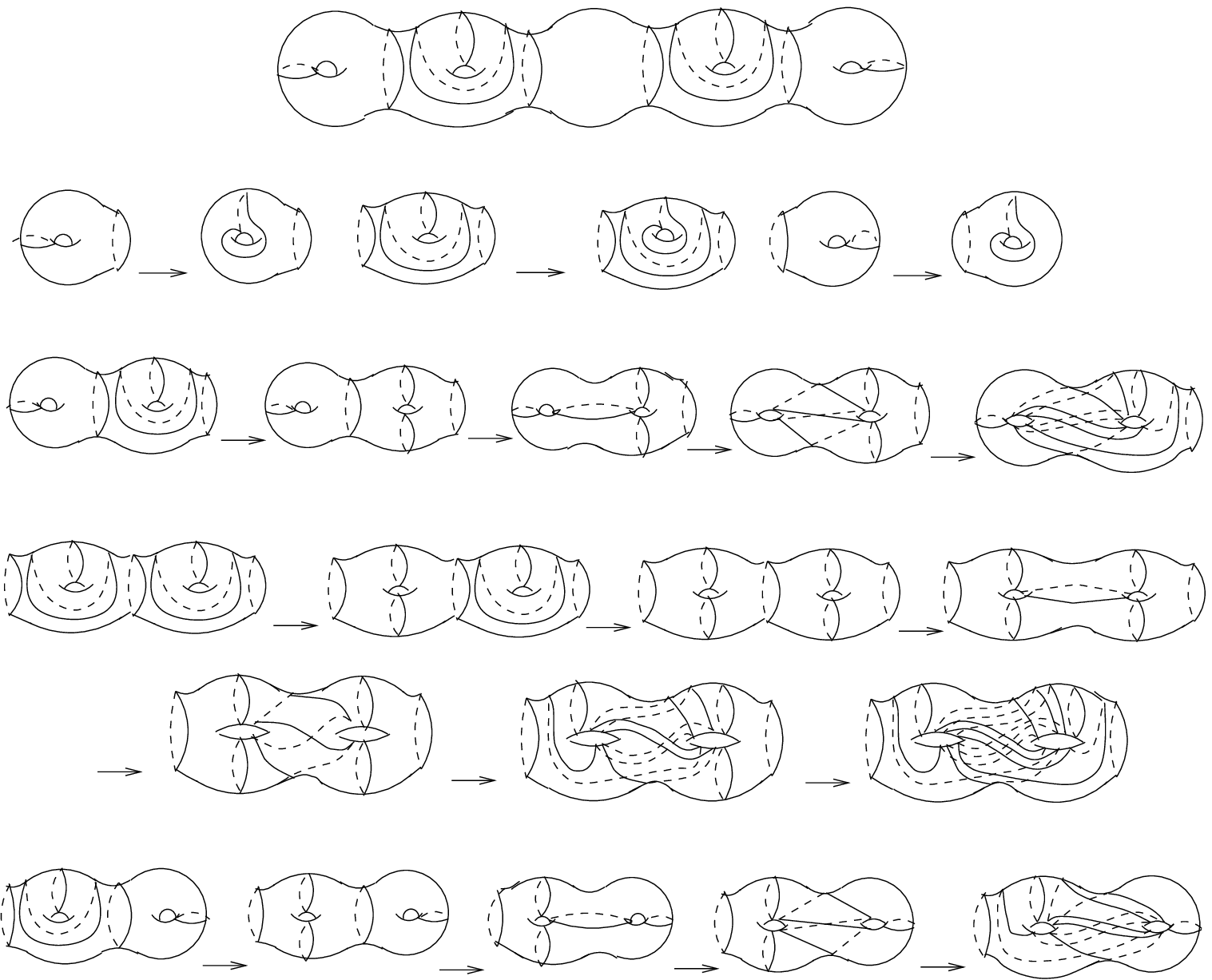

\caption{Paths From $\Ols$}  \label{loopspath}
\end{center}
\end{figure}
These paths can be seen in Figure~\ref{loopspath}, where only the piece of $\Ols$ that changes in the pants decompositions along the path from $\Ols$ to $T_{c_{i}}(\Ols)$ is shown.  Putman has constructed paths from another pants decomposition that is in a different orbit.  They obviously differ from ours, though they inspired this construction~\cite{P}. We now use our path lengths as follows.
\begin{lemma} \label{lem1}
Let $w \in \Modgsigma$, so there are some standard Lickorish generating curves $\{c_{i_{1}}, \ldots, c_{i_{k}} \}$ such that $w=T^{e_{i_{1}}}_{c_{i_{1}}} \cdots T^{e_{i_{k}}}_{c_{i_{k}}}$ with $e_{i_{j}}=\pm 1 $ for $1 \leq j \leq k$.  Then there is a path from $\Ols$ to $w(\Ols)$ of length equal to $\sum_{j=1}^{k}l_{c_{i_{j}}}$, where $l_{c_{i_{j}}}$ is the length of a path from $\Ols$ to both $T_{c_{i_{j}}}(\Ols)$ and  $T^{-1}_{c_{i_{j}}}(\Ols)$.
\end{lemma}
\begin{proof}
Consider the path from $\Ols$ to $T^{e_{i_{j}}}_{c_{i_{j}}}(\Ols)$.  Let $f \in \Modgsigma$.  Apply $f$ to every pants decomposition in the path.  Since intersection numbers are preserved, we get a path from $f(\Ols)$ to $f (T^{e_{i_{j}}}_{c_{i_{j}}}(\Ols))=fT^{e_{i_{j}}}_{c_{i_{j}}}(\Ols)$ that is the same length, $l_{c_{i_{j}}}$, as the path from $\Ols$ to $T^{e_{i_{j}}}_{c_{i_{j}}}(\Ols)$.  Consequently, there will be a path from $T^{e_{i_{1}}}_{c_{i_{1}}} \cdots T^{e_{i_{j}}}_{c_{i_{j}}}(\Ols)$ to $T^{e_{i_{1}}}_{c_{i_{1}}} \cdots T^{e_{i_{j}}}_{c_{i_{j}}}T^{e_{i_{j+1}}}_{c_{i_{j+1}}}(\Ols)$ that will be the same length as a path from $\Ols$ to $T^{e_{i_{j+1}}}_{c_{i_{j+1}}}(\Ols)$.  Thus we have a path 
\bdm
\Ols \rightarrow T^{e_{i_{1}}}_{c_{i_{1}}}(\Ols) \rightarrow T^{e_{i_{1}}}_{c_{i_{1}}}(T^{e_{i_{2}}}_{c_{i_{2}}}(\Ols)) \rightarrow \ldots \rightarrow T^{e_{i_{1}}}_{c_{i_{1}}} \cdots T^{e_{i_{k-1}}}_{c_{i_{k-1}}}(T^{e_{i_{k}}}_{c_{i_{k}}}(\Ols))
\edm
whose length is $\sum_{j=1}^{k}l_{c_{i_{j}}}$.  Since $w=T^{e_{i_{1}}}_{c_{i_{1}}} \cdots T^{e_{i_{k-1}}}_{c_{i_{k-1}}}T^{e_{i_{k}}}_{c_{i_{k}}}$,  we have a path from $\Ols$ to $w(\Ols)$ whose length is $\sum_{j=1}^{k}l_{c_{i_{j}}}$. 
\end{proof}
\begin{lemma} \label{lem2}
Let $\{T_{c_{1}},\ldots,T_{c_{3g-1}} \}$ be the standard Lickorish generators.  Then for any $f \in \Modgsigma$, $\{T_{f(c_{1})},\ldots,T_{f(c_{3g-1})}\}$ is another set of Lickorish generators for the mapping class group.
\end{lemma}
\begin{proof}
Let $h \in \Modgsigma$.  Then $f^{-1}hf \in \Modgsigma$.  Therefore there are some standard Lickorish generating curves $\{c_{i_{1}},\ldots, c_{i_{k}} \}$ such that $f^{-1}hf=T^{e_{i_{1}}}_{c_{i_{1}}} \cdots T^{e_{i_{k}}}_{c_{i_{k}}}$ with $e_{i_{j}}=\pm 1 $ for $1 \leq j \leq k$.  It is a basic fact about Dehn twists that $T_{f(a)} = fT_{a}f^{-1}$ for all $f \in \Modgsigma$~\cite{FM}, so $f^{-1}T^{e_{i_{j}}}_{f(c_{i_{j}})}f = T^{e_{i_{j}}}_{c_{i_{j}}}$ for $1 \leq j \leq k$.  Thus, 
\begin{eqnarray*}
f^{-1}hf  & = &  f^{-1}T^{e_{i_{1}}}_{f(c_{i_{1}})}ff^{-1}T^{e_{i_{2}}}_{f(c_{i_{2}})}f \cdots f^{-1}T^{e_{i_{k}}}_{f(c_{i_{k}})}f \\
              & = & f^{-1}T^{e_{i_{1}}}_{f(c_{i_{1}})}T^{e_{i_{2}}}_{f(c_{i_{2}})} \cdots T^{e_{i_{k}}}_{f(c_{i_{k}})}f.
\end{eqnarray*}
Hence, $h = T^{e_{i_{1}}}_{f(c_{i_{1}})}T^{e_{i_{2}}}_{f(c_{i_{2}})} \cdots T^{e_{i_{k}}}_{f(c_{i_{k}})}$, so $\{T_{f(c_{1})},\ldots,T_{f(c_{3g-1})}\}$ also generates $\Modgsigma$. 
\end{proof}
\begin{prop} \label{prop}
Let $\Ra, \Rbb \in \Oloops$ and $w,h \in \Modgsigma$ such that $h(\Ra)=\Ols$ and $w(\Ra)=\Rbb$.  For some subset of the standard Lickorish generators, $\{c_{i_{1}}, \ldots, c_{i_{k}} \}$, we have $w = T^{e_{i_{1}}}_{h^{-1}(c_{i_{1}})} \cdots T^{e_{i_{k}}}_{h^{-1}(c_{i_{k}})}$, with $e_{i_{j}}=\pm 1 $ for $1 \leq j \leq k$, where this is a minimum length word in these twists.  There is a path from $\Ra$ to $\Rbb$ in $\Oloops$ of length equal to $\sum_ {1 \leq i_{j} \leq g}|e_{i_{j}}| + 4\sum_ {i_{j}=g+1,\, 2g-1}|e_{i_{j}}| + 6\sum_ {g+2 \leq i_{j} \leq 2g-2}|e_{i_{j}}|$    
\end{prop}
\begin{proof}
By Lemma~\ref{lem2}, $\{T_{h^{-1}(c_{1})},\ldots,T_{h^{-1}(c_{3g-1})}\}$ also generates $\Modgsigma$, so $w$ can be expressed in terms of these generators.  We can assume that $w$ is a minimal length word in these generators since Mosher~\cite{MM} has shown that the mapping class group of closed surfaces is automatic, ensuring that the word problem is solvable for it.  Since $h(\Ra), h(\Rbb) \in \Oloops$, there must be some $f \in \Modgsigma$ whereby $f(h(\Ra)) = h(\Rbb)$.  Then
\begin{eqnarray*}
f(h(\Ra)) & = & h(w(\Ra)) \\
           & = & h(T^{e_{i_{1}}}_{h^{-1}(c_{i_{1}})} \cdots T^{e_{i_{k}}}_{h^{-1}(c_{i_{k}})}(\Ra)) \\
           & = & T^{e_{i_{1}}}_{hh^{-1}(c_{i_{1}})} \cdots T^{e_{i_{k}}}_{hh^{-1}(c_{i_{k}})}(h(\Ra)) \\
           & = & T^{e_{i_{1}}}_{c_{i_{1}}} \cdots T^{e_{i_{k}}}_{c_{i_{k}})}(h(\Ra)),
\end{eqnarray*}
so  $T^{e_{i_{1}}}_{c_{i_{1}}} \cdots T^{e_{i_{k}}}_{c_{i_{k}})}$ is a word in the standard Lickorish generators that takes $\Ols$ to $h(\Rbb)$.  Lemma~\ref{lem1} tells us there is a path from $\Ols$ to $h(\Rbb)$ of length $\sum_{j=1}^{k}l_{c_{i_{j}}}$.  Substitute the values for $l_{c_{i_{j}}}$ in Equation~\ref{lc}, to get a path of length equal to $\sum_ {1 \leq i_{j} \leq g}|e_{i_{j}}| + 4\sum_ {i_{j}=g+1,\, 2g-1}|e_{i_{j}}| + 6\sum_ {g+2 \leq i_{j} \leq 2g-2}|e_{i_{j}}|$.  We now apply $h^{-1}$ to every pants decomposition in this path to get a path of the same length from $h^{-1}(\Ols) = h^{-1}(h(\Ra)) = \Ra$ to $h^{-1}(h(\Rbb))= \Rbb$. 
\end{proof}
\begin{proof}[Proof of Theorem~\ref{goal}]
From our work on orbits we know that corresponding to an orbit path from $\pao$(resp. $\pbo$) to $\Oloops$ will be a path of the same length in the pants graph from $\pa$(resp. $\pb$) to some $\Ra$(resp. $\Rbb$) in $\Oloops$. Let $w,h \in \Modgsigma$ where $w\Ra = \Rbb$ and $h\Ra = \Ols$.  By Proposition~\ref{prop}, there is some subset of the standard Lickorish generators, $\{c_{i_{1}}, \ldots, c_{i_{k}} \}$, such that $w = T^{e_{i_{1}}}_{h^{-1}(c_{i_{1}})} \cdots T^{e_{i_{k}}}_{h^{-1}(c_{i_{k}})}$, with $e_{i_{j}}=\pm 1 $for $1 \leq j \leq k$, where this is a minimum length word in these twists and there is a path from $\Ra$ to $\Rbb$ in $\Oloops$ of length equal to $\sum_ {1 \leq i_{j} \leq g}|e_{i_{j}}| + 4\sum_ {i_{j}=g+1,\, 2g-1}|e_{i_{j}}| + 6\sum_ {g+2 \leq i_{j} \leq 2g-2}|e_{i_{j}}|$.  This plus Equation~\ref{diamglp} gives us 
\begin{equation} \nonumber
   \DgPaPb  \leq \left\{\begin{array}{ll} 
                 4\log{(g-1)!} + 4g - 6 +\sum_{1 \leq i_{j} \leq g}|e_{i_{j}}| + \\
\bbb \bbb \bbb  4\sum_ {i_{j}=g+1,\, 2g-1}|e_{i_{j}}| + 6\sum_ {g+2 \leq i_{j} \leq 2g-2}|e_{i_{j}}| & \mbox{$2 \leq g \leq 5$}\\
               4\log{(g-1)!} + 6g - 16 +\sum_{1 \leq i_{j} \leq g}|e_{i_{j}}| + \\
\bbb \bbb \bbb  4\sum_ {i_{j}=g+1,\, 2g-1}|e_{i_{j}}| + 6\sum_ {g+2 \leq i_{j} \leq 2g-2}|e_{i_{j}}| & \mbox{$6 \leq g$}
                 \end{array}
                \right.
\end{equation}
\end{proof}
\section{Open Problems and Limitations} \label{op}
It would be interesting to investgate the structure of the orbit graph.  If we could document the length of a path between any two elements of any orbit, so as not to have to go through $\Oloops$, we might have a shorter path and thus a better upper bound.  We also don't know if the shortest word operating on one pants decomposition, taking it to another, will yield the shortest path between these two pants decompositions.  It would be nice to extend this work to surfaces with boundaries.
\section*{Acknowledgement}
I want to thank Joan Birman for her interest in and time devoted to discussions about this subject.\\
\\
moser@math.columbia.edu\\
74 Wensley Drive\\
Great Neck, N.Y. 11020
 
\flushleft
\bibliography{pantsdistancenew4_18bb}
\bibliographystyle{plain}
\end{document}